\documentclass[letterpaper,10pt]{amsart}
\usepackage{thmtools}
\usepackage{amssymb}
\usepackage{xcolor}
\usepackage{hyperref}
\usepackage{dsfont}
\usepackage{enumerate}
\usepackage{mathrsfs}
\usepackage{mathtools}

\usepackage[parfill]{parskip}

\definecolor{carminepink}{rgb}{0.92, 0.3, 0.26}

\hypersetup{backref,colorlinks=true,allcolors=carminepink} 

\declaretheorem[parent=section]{theorem}

\declaretheorem[sibling=theorem]{lemma}
\declaretheorem[sibling=theorem]{corollary}

\declaretheorem[sibling=theorem,style=remark]{remark}

\declaretheorem[sibling=theorem,style=definition]{definition}

\DeclareSymbolFont{bbold}{U}{bbold}{m}{n}
\DeclareSymbolFontAlphabet{\mathbbold}{bbold}

\definecolor{electricultramarine}{rgb}{0.25, 0.0, 1.0}

\providecommand{\norm}[1]{\left\Vert#1\right\Vert}

\newcommand{\cont}[1]{\stackrel{#1}{\frown}}
\newcommand{\conts}[1]{\stackrel{#1}{\frown}}

\newcommand{\R}{\mathbb{R}}

\newtheorem*{theorem*}{Theorem}
\newtheorem*{acknow*}{Acknowledgments}

\begin{document}

\author{Solesne Bourguin}
\address{Boston University, Department of Mathematics and Statistics, 111
  Cummington Mall, Boston, MA 02215, USA}
\email{solesne.bourguin@gmail.com}
\author{Simon Campese}
\address{University of Luxembourg, Mathematics Research Unit, 6, rue Richard
  Coudenhove-Kalergi, 1359 Luxembourg, Luxembourg}
\email{simon.campese@uni.lu}
\title[Free quantitative Fourth Moment Theorems]{Free quantitative Fourth Moment
  Theorems on Wigner space}  
\begin{abstract}
We prove a quantitative Fourth Moment Theorem for Wigner integrals of any order
with symmetric kernels, generalizing an earlier result from Kemp et al. (2012). The proof relies on free stochastic analysis and
uses a new biproduct formula for bi-integrals. A consequence of our main result
is a Nualart-Ortiz-Latorre type characterization of convergence in law to the
semicircular distribution for Wigner integrals. As an application, we provide
Berry-Esseen type bounds in the context of the free Breuer-Major theorem for the
free fractional Brownian motion.
\end{abstract}
\maketitle

\section{Introduction}
Let $\left( \mathscr{A}, \varphi \right) $ be a tracial $W^*$-probability space,
$S$ be a semicircular random variable and $F = I_n(f)$ be a self-adjoint Wigner
integral (for a simple example, take off-diagonal homogeneous sums of a
semicircular system). Recently, Kemp et al. showed in~\cite{kemp_wigner_2012}
that for a sequence of such Wigner integrals, convergence of the fourth moment
controls convergence in distribution towards the semicircular law. Moreover,
they provided a quantitative bound in terms of the free gradient operator, which
is of the form  (all unexplained notation appearing in this section will be
introduced in the sequel)
\begin{equation}
\label{wignercasebound}
d_{\mathcal{C}_2}\left(F ,S \right) \leq \frac{1}{2} \varphi \otimes \varphi \left( \left | \int \nabla_s\left( N_0^{-1}F\right) \sharp \left( \nabla_s F\right) ^{*}ds -1\otimes 1 \right | \right).
\end{equation}
Here, $d_{\mathcal{C}_2}$ is a distance that metrizes free convergence in
distribution (see Definition \ref{dc2distance}), $\nabla$ denotes the free
gradient operator first introduced by Biane and Speicher in
\cite{biane_stochastic_1998} and $N_{0}^{-1}$ stands for the pseudo-inverse of
the number operator (see Section \ref{prelim}). In the special case of Wigner
integrals of order two, Kemp et al. showed in~\cite{kemp_wigner_2012} that the
gradient expression 
appearing in \eqref{wignercasebound} can further be bounded by the fourth
moment. To be more precise, it holds that
\begin{equation}
\label{I2bound}
\begin{aligned} d_{\mathcal{C}_2}\left(I_{2}(f) ,S \right) & \leq \frac{1}{2} \varphi \otimes \varphi \left( \left | \int \nabla_s\left( N_0^{-1}I_{2}(f)\right) \sharp \left( \nabla_s I_{2}(f)\right) ^{*}ds -1\otimes 1 \right | \right) \\
& \leq\frac{1}{2} \sqrt{\frac{3}{2}}\sqrt{\varphi\left(I_{2}(f)^4 \right)-2 }.\end{aligned}
\end{equation}
A question left open in the aforementioned article is whether a similar fourth
moment bound holds for Wigner integrals of higher orders, as is the case in the
commutative setting (see Nualart and Peccati \cite{nualart_central_2005} and
Nourdin and Peccati\cite{nourdin_steins_2009}). In this paper, we provide a
positive answer to this question by proving fourth moment bounds for Wigner
integrals of any order with symmetric kernels. Our main result can be
paraphrased as follows (see Theorem \ref{maintheorem1} for a precise
statement).
\begin{theorem*}
For a Wigner integral $F$ of order $n$ with normalized symmetric kernel it holds
that 
\begin{equation*}
\varphi \otimes \varphi\left( \left | \int_{\mathbb{R}_{+}}\nabla_s\left(
      N_0^{-1} F\right) \sharp \left(\nabla_s F \right)^{*}  ds- 1\otimes 1
  \right |^2\right) \leq C_n \Big(\varphi\left( F^4\right) -2\Big). 
\end{equation*}
\end{theorem*}
The constant $C_n$ grows asymptotically linearly with $n$ and is a local maximum
of a certain polynomial (see Theorem~\ref{maintheorem1} for full
details). Combined with \eqref{wignercasebound}, our result quantifies the free
Fourth Moment Theorem \cite[Theorem 1.3]{kemp_wigner_2012} for the case of
Wigner integrals with symmetric kernels. In particular, $\sqrt{C_2} =
\frac{1}{2} \sqrt{\frac{3}{2}}$, so that, by Cauchy-Schwarz, the
bound~\eqref{I2bound} is included as a special case.

It is well-known that in order to ensure that Wigner integrals are self-adjoint
(and thus free random variables), the symmetry of the kernel can be relaxed to
mirror-symmetry (see Definition~\ref{mirrorsymdef}). As our main bound is
stated for symmetric kernels, the natural question arises whether or not it can
be generalized to cover the mirror-symmetric case as well. The answer to this
question is negative, as is shown by the counterexample in Remark
\ref{bigremark}.  

In the proof of our main result we use a new biproduct formula (see Theorem
\ref{productformulabi}) for Wigner bi-integrals (see Subsection
\ref{subsectionbiintegrals}) which generalizes the product formula proved by
Biane and Speicher in \cite{biane_stochastic_1998} for usual Wigner
integrals. In this biproduct formula, the nested contractions become what we
call bicontractions. As product formulae play a central role in free (and also
classical) stochastic analysis, this might be of independent interest. Other
ingredients include the free Malliavin calculus introduced by
Biane and Speicher in~\cite{biane_stochastic_1998} as well as a fine
combinatorial analysis. A direct consequence of our bound is a
Nualart-Ortiz-Latorre type equivalent condition for convergence towards the
semicircular law, which reads as follows (see Theorem
\ref{nualartortizlatorrefree} for a precise statement).
\begin{theorem*}
A sequence $F_k$ of Wigner integrals of order $n$ with normalized symmetric
kernels converges in law to the standard semicircular distribution if, and only
if, 
\begin{equation*}
\int_{\mathbb{R}_{+}}\left( \nabla_s F_k\right)  \sharp \left(\nabla_s F_k
\right)^{*}ds \rightarrow n \cdot 1 \otimes 1\ \text{ in }\
L^2\left(\mathscr{A}\otimes\mathscr{A}, \varphi \otimes \varphi \right). 
\end{equation*}
\end{theorem*}
This is a free analogue of the main result of \cite{nualart_central_2008}.

Our findings contribute to the growing literature on free limit theorems
obtained by means of free Malliavin calculus and free stochastic
analysis. Earlier results include the already mentioned free Fourth Moment
Theorem for multiple Wigner integrals~\cite{kemp_wigner_2012}, its
multidimensional extension~\cite{nourdin_multi-dimensional_2013}, the free
Fourth Moment Theorem for free Poisson multiple integrals proved
in~\cite{bourguin_semicircular_2014} and~\cite{bourguin_vector-valued_2016},
free non-central limit theorems for Wigner and free Poisson integrals obtained
in~\cite{deya_convergence_2012}, \cite{nourdin_poisson_2013}
and~\cite{bourguin_poisson_2015}, as well as limit theorems for the $q$-Brownian
motion~\cite{deya_fourth_2013} and convergence of free
processes~\cite{nourdin_central_2014}. However, all these results,
with the exception of~\cite{kemp_wigner_2012} for the case of second order
Wigner integrals, are not quantitative. In the commutative setting, which
inspired this line of research in the context of free probability theory, the
picture is much more complete. Here, quantitative limit theorems exist in the
framework of Wiener integrals (\cite{nualart_central_2005,peccati_gaussian_2005,nualart_central_2008,nourdin_steins_2009,nourdin_multivariate_2010,nourdin_noncentral_2009}
and references therein), Poisson integrals
(\cite{peccati_steins_2010,peccati_multi-dimensional_2010,peccati_chen-stein_2011,bourguin_portmanteau_2014,peccati_gamma_2013}
and references therein) and eigenfunctions of diffusive Markov generators
(\cite{ledoux_chaos_2012,azmoodeh_fourth_2014,campese_multivariate_2016}).

The rest of this paper is organized as follows: Section~\ref{prelim} introduces
the basic concepts of free probability theory and free stochastic analysis. The
biproduct formula, our fourth moment bound, as well as the Nualart-Ortiz-Latorre
characterization are presented and proved in Section~\ref{mainressections}. We
conclude by providing a Berry-Esseen bound for the free Breuer-Major theorem for
the free fractional Brownian motion in Section~\ref{breuermajorsection}.

\section{Preliminaries}
\label{prelim}
\subsection{Elements of free probability}
In the following, a short introduction to free probability theory is
provided. For a thorough and complete treatment, see \cite{nica_lectures_2006},
\cite{voiculescu_free_1992} and \cite{hiai_semicircle_2000}. Let $\left(
  \mathscr{A}, \varphi \right) $ be a tracial $W^*$-probability space, that is
$\mathscr{A}$ is a von Neumann algebra with involution $*$ and $\varphi \colon
\mathscr{A} \rightarrow \mathbb{C}$ is a unital linear functional assumed to be
weakly continuous, positive (meaning that $\varphi\left( X\right) \geq 0$
whenever  $X$ is a non-negative element of $\mathscr{A}$), faithful (meaning
that $\varphi\left(XX^* \right) = 0 \Rightarrow X = 0$ for every $X \in
\mathscr{A}$) and tracial (meaning that $\varphi\left(XY \right) =
\varphi\left(YX \right) $ for all $X,Y \in \mathscr{A}$). The self-adjoint
elements of $\mathscr{A}$ will be referred to as random variables. Given a
random variable $X \in \mathscr{A}$, the law of $X$ is defined to be the unique
Borel measure on $\R$ having the same moments as $X$ (see~\cite[Proposition
3.13]{nica_lectures_2006}). The non-commutative space 
$L^2(\mathscr{A},\varphi)$ denotes the completion of $\mathscr{A}$ with respect
to the norm $\norm{X}_2 = \sqrt{\varphi\left( XX^* \right) }$. 
\begin{definition}
A collection of random variables $X_1, \ldots , X_n$ on $\left( \mathscr{A}, \varphi \right) $ is said to be free if $$\varphi\left( \left[P_1\left(X_{i_1}\right) - \varphi\left( P_1\left(X_{i_1}\right)\right)   \right] \cdots \left[P_m\left(X_{i_m}\right) - \varphi\left( P_m\left(X_{i_m}\right)\right)   \right] \right) = 0 $$ whenever $P_1, \ldots , P_m$ are polynomials and $i_1, \ldots , i_m \in \left\lbrace 1, \ldots, n\right\rbrace $ are indices with no two adjacent $i_j$ equal.
\end{definition}
 Let $X\in \mathscr{A}$. The $k$-th moment of $X$ is given by the quantity
 $\varphi(X^{k})$, $k \in \mathbb{N}_0$. Now assume that $X$ is a
 self-adjoint bounded element of $\mathscr{A}$ (in other words, $X$ is
 a bounded random variable), and write $\rho(X)= \norm{X} \in [0, \infty)$ to indicate the {\it spectral radius} of $X$. 
\begin{definition}
The {\it law} (or {\it spectral measure}) of $X$ is defined as the
unique Borel probability measure $\mu_{X}$ on the real line such that $\int_{\mathbb{R}}P(t)\ d\mu_{X}(t) = \varphi(P(X))$
for every polynomial $P \in \mathbb{R}\left[ X\right]$. A consequence
of this definition is that $\mu_X$ has support in $[-\rho(X), \rho(X)]$.

\end{definition}
The existence and uniqueness of $\mu_X$ in such a general framework are proved e.g. in \cite[Theorem 2.5.8]{tao_topics_2012} (see also \cite[Proposition 3.13]{nica_lectures_2006}). Note that, since $\mu_X$ has compact support, the measure $\mu_X$ is completely determined by the sequence $\left\lbrace \varphi(X^k) \colon k\geq 1\right\rbrace $. 

Let $\left\lbrace X_{n} \colon n \geq 1\right\rbrace $ be a sequence of non--commutative random variables, each possibly belonging to a different non-commutative probability space $(\mathscr{A}_n, \varphi_n)$. 
\begin{definition}
The sequence $\left\lbrace X_{n} \colon n \geq 1\right\rbrace $ is said to converge in distribution to a limiting non-commutative random variable $X_{\infty}$ (defined on $(\mathscr{A}_\infty, \varphi_\infty)$), if $\lim_{n \to+\infty}\varphi_n(P(X_{n})) = \varphi_\infty(P(X_{\infty}))$ for every polynomial $P\in \R[X]$.
\end{definition}
 If $X_{n}, X_\infty$ are bounded (and therefore the spectral measures $\mu_{X_n}, \mu_{X_\infty}$ are well-defined), this last relation is equivalent to saying that $$\int_\R P(t)\, \mu_{X_n}(dt) \to \int_\R P(t)\, \mu_{X_\infty}(dt).$$ An application of the method of moments yields immediately that, in this case, one has also that $\mu_{X_n}$ weakly converges to $\mu_{X_\infty}$, that is $\mu_{X_n}(f) \to \mu_{X_\infty}(f)$, for every $f: \R\to \R$ bounded and continuous (note that no additional uniform boundedness assumption is needed). 

Let $h(x) = \int_{\R}e^{ix\xi}\nu(d\xi)$ be the Fourier transform of a complex measure $\nu$ on $\R$. Note that, as $\nu$ is finite, $h$ is continuous and bounded. For such functions $h$, define the seminorm $\mathscr{I}_2(h)$ by 
\begin{equation*}
\mathscr{I}_2(h) = \int_{\R}\xi^2 \vert \nu \vert (d\xi).
\end{equation*}
Let $\mathcal{C}_2$ denote the set of those functions $h$ for which $\mathscr{I}_2(h) < \infty$. Using the seminorm $\mathscr{I}_2$ and the set of functions $\mathcal{C}_2$, one can define a distance between two self-adjoint random variables.
\begin{definition}
\label{dc2distance}
For two self-adjoint random variables $X,Y$, the distance $d_{\mathcal{C}_2}(X,Y)$ between $X$ and $Y$ is defined as 
\begin{equation*}
d_{\mathcal{C}_2}(X,Y) = \sup \left\lbrace \left| \varphi(h(X)) - \varphi(h(Y)) \right| \colon h \in \mathcal{C}_2,\ \mathscr{I}_2(h) \leq 1\right\rbrace.
\end{equation*}
\end{definition}
 As is proved in \cite{kemp_wigner_2012}, the distance $d_{\mathcal{C}_2}$ is
 weaker than the Wasserstein distance but still metrizes convergence in law. 

\begin{definition}
The centered semicircular distribution with variance $t>0$, denoted by $\mathcal{S}(0,t)$, is the probability distribution given by $$\mathcal{S}(0,t)(dx) = (2\pi t)^{-1}\sqrt{4t-x^2}dx, \quad |x|< 2\sqrt{t}.$$ 
\end{definition}

\begin{definition}
A free Brownian motion $S$ consists of: (i) a filtration $\left\lbrace \mathscr{A}_t \colon t \geq 0 \right\rbrace $ of von Neumann sub-algebras of $\mathscr{A}$ (in particular, $\mathscr{A}_s \subset \mathscr{A}_t$ for $0 \leq s < t$), (ii) a collection $S = \left\lbrace S_t \colon t\geq 0\right\rbrace $ of self-adjoint operators in $\mathscr{A}$ such that: (a) $S_0 = 0$ and $S_t \in \mathscr{A}_t$ for all $t \geq 0$, (b) for all $t \geq 0$, $S_t$ has a semicircular distribution with mean zero and variance $t$, and (c) for all $0 \leq u < t$, the increment $S_t - S_u$ is free with respect to $\mathscr{A}_u$, and has a semicircular distribution with mean zero and variance $t-u$.
\end{definition}

 For every integer $n\geq 1$, the space $L^2\left( \R_{+}^n;\mathbb{C}\right) = L^2\left( \R_{+}^n\right)$ denotes the collection of all complex-valued functions on $\R_{+}^n$ that are square-integrable with respect to the Lebesgue measure on $\R_{+}^n$. 
\begin{definition}
\label{mirrorsymdef}
Let $n$ be a natural number and let $f$ be a function in $L^2\left( \R_{+}^n\right)$. 
\begin{enumerate}
\item The adjoint of $f$ is the function $f^{\ast}\left(t_1, \ldots , t_n \right) = \overline{f\left(t_n, \ldots , t_1\right)}$.
\item The function $f$ is called mirror-symmetric if $f = f^{\ast}$, i.e., if $$f\left(t_1, \ldots , t_n \right) = \overline{f\left(t_n, \ldots , t_1\right)}$$ for almost all $\left( t_1,\ldots , t_n\right) \in \R_{+}^{n}$ with respect to the product Lebesgue measure.
\item The function $f$ is called (fully) symmetric if it is real-valued and, for any permutation $\sigma$ in the symmetric group $\mathfrak{S}_{n}$, it holds that $f\left( t_1, \ldots , t_n\right) = f\left( t_{\sigma(1)}, \ldots , t_{\sigma(n)}\right) $ for almost all $\left( t_1,\ldots , t_n\right) \in \R_{+}^{n}$ with respect to the product Lebesgue measure.
\end{enumerate}
\end{definition}
\begin{definition}
\label{defcontractions}
Let $n,m$ be natural numbers and let $f \in L^2\left( \R_{+}^n\right)$ and $g \in L^2\left( \R_{+}^m\right)$. Let $p \leq n \wedge m$ be a natural number. The $p$-th nested contraction $f \cont{p} g$ of $f$ and $g$ is the $L^2\left( \R_{+}^{n+m-2p}\right)$ function defined by nested integration of the middle $p$ variables in $f \otimes g$:
\begin{eqnarray*}
 f  \cont{p} g (t_1,\ldots, t_{n+m - 2p}) &=& \int_{\R_{+}^{p}}f(t_1,\ldots , t_{n-p},s_1, \ldots , s_p) \\
&& \qquad\qquad g(s_p, \ldots , s_1 , t_{n-p+1},\ldots, t_{n+m-2p})ds_1 \cdots ds_p.
\end{eqnarray*}
In the case where $p=0$, the function $f  \cont{0} g$ is just given by $f \otimes g$. 
\end{definition}
 For $f \in L^2\left(\R_{+}^n\right)$, we denote by $I_{n}(f)$ the multiple
 Wigner integral of $f$ with respect to the free Brownian motion as introduced
 in \cite{biane_stochastic_1998}. The space $L^2(\mathcal{S}, \varphi) =
 \{I_n(f) : f\in L^2(\R_{+}^n), n\geq 0\}$ is a unital $\ast$-algebra, with
 product rule given, for any $n,m\geq 1$, $f \in L^2\left(\R_{+}^n\right)$, $g
 \in L^2\left(\R_{+}^m\right)$, by 
\begin{equation}
\label{productformula}
I_n(f)I_m(g) = \sum_{p=0}^{n \wedge m} I_{n+m-2p}\left( f \cont{p} g\right) 
\end{equation}
and involution $I_n(f)^{\ast} = I_n(f^{\ast})$. For a proof of this formula, see \cite{biane_stochastic_1998}. Furthermore, as is well-known, multiple integrals of different orders are orthogonal in $L^2(\mathscr{A},\varphi)$, whereas for two integrals of the same order, the Wigner isometry
\begin{equation}
\label{wigneriso}
\varphi\left( I_n(f)I_n(g)^{*}\right) = \left\langle f,g \right\rangle_{L^2\left(\mathbb{R}_{+}^n \right) }.
\end{equation}
holds.
\begin{remark}
Observe that it follows from the definition of the involution on the algebra $L^2(\mathcal{S}, \varphi)$ that operators of the type $I_n(f)$ are self-adjoint if and only if $f$ is mirror-symmetric.
\end{remark}
\subsection{Bi-integrals and free gradient operator}
\label{subsectionbiintegrals}
This subsection introduces the notion of bi-integral and the action of the free
gradient operator on Wigner integrals.  For a full treatment of these objects, see~\cite{biane_stochastic_1998}. 

 Let $n,m$ be two positive integers and $f=g\otimes h \in L^2\left( \mathbb{R}_{+}^{n}\right) \otimes L^2\left( \mathbb{R}_{+}^{m}\right)$. Then, the Wigner bi-integral $I_n \otimes I_m(f)$ is defined as 
\begin{equation*}
I_n \otimes I_m(f) = I_n(g) \otimes I_m(h).
\end{equation*}
This definition is extended linearly to generic elements $f \in L^2\left(
  \mathbb{R}_{+}^{n}\right) \otimes L^2\left( \mathbb{R}_{+}^{m}\right) \cong
L^2\left( \mathbb{R}_{+}^{n+m}\right)$. From the Wigner isometry
\eqref{wigneriso} for multiple integrals, we obtain the so called Wigner bisometry: for $f \in L^2\left( \mathbb{R}_{+}^{n}\right) \otimes L^2\left( \mathbb{R}_{+}^{m}\right)$ and $g \in L^2\left( \mathbb{R}_{+}^{n'}\right) \otimes L^2\left( \mathbb{R}_{+}^{m'}\right)$ it holds that
\begin{equation}
\label{wignerbisometry}
\varphi \otimes \varphi\left(I_n \otimes I_m(f) I_{n'} \otimes I_{m'}(g)^{*} \right)= \begin{cases} \left\langle f,g \right\rangle_{L^2\left( \mathbb{R}_{+}^{n}\right) \otimes L^2\left( \mathbb{R}_{+}^{m}\right)} & \quad \text{if $n=n'$ and $m=m'$},\\ 0 & \quad \text{otherwise}\end{cases} 
\end{equation}
\begin{remark}
Observe that, for any natural numbers $n,m$ and any function $g\otimes h \in  L^2\left( \mathbb{R}_{+}^{n}\right) \otimes L^2\left( \mathbb{R}_{+}^{m}\right)$, it holds that 
\begin{align*}
I_{n} \otimes I_{m}\left(g\otimes h\right)^{*} & = \left( I_{n}\left(g \right)  \otimes I_{m}\left( h\right)\right)^{*} = I_{n}\left(g \right)^{*}  \otimes I_{m}\left( h\right)^{*} \\ & = I_{n}\left(g^{*} \right)  \otimes I_{m}\left(h ^{*}\right) = I_{n} \otimes I_{m}\left(\left( g\otimes h\right) ^{*}\right),
\end{align*}
so that the operator $I_{n}\otimes I_{m}\left(g\otimes h\right)$ is
self-adjoint if and only if both the function $g$ and $h$ are mirror-symmetric. By continuous extension (using the Wigner
bisometry \eqref{wignerbisometry}), it holds that for any fully symmetric function $f \in L^2\left( \mathbb{R}_{+}^{n}\right) \otimes L^2\left( \mathbb{R}_{+}^{m}\right)$, the operator $I_{n} \otimes I_{m}\left(f\right)$ is self-adjoint.
\end{remark}
Let $\left( \mathscr{A},\varphi\right) $ be a $W^{*}$-probability space. An $\mathscr{A} \otimes \mathscr{A}$-valued stochastic process $t\mapsto U_t$ is called a biprocess. For $p \geq 1$, $U$ is an element of $\mathscr{B}_{p}$, the space of $L^p$-biprocesses, if its norm
\begin{equation*}
\norm{U}_{\mathscr{B}_p}^{2} = \int_{0}^{\infty} \norm{U_t}_{L^p\left(\mathscr{A} \otimes \mathscr{A}, \varphi \otimes \varphi\right) }^{2}dt
\end{equation*}
is finite. 

 The free gradient operator $\nabla \colon L^2\left(\mathcal{S},\varphi \right) \rightarrow \mathscr{B}_{2} $ is a densely-defined and closable operator whose action on Wigner integrals is given by
\begin{equation*}
\nabla_t I_n(f) = \sum_{k=1}^{n}I_{k-1} \otimes I_{n-k}\left(f_t^{(k)} \right),
\end{equation*}
where $f_t^{(k)}(x_1,\ldots,x_{n-1}) = f(x_1,\ldots, x_{k-1},t,x_{k},\ldots, x_{n-1})$ is viewed as an element of $L^2\left(\mathbb{R}_{+}^{k-1} \right) \otimes L^2\left(\mathbb{R}_{+}^{n-k} \right)$.
\begin{remark}
For general elements of $L^2\left(\mathcal{S},\varphi \right)$ in its domain, the free gradient is customarily defined via a Fock space construction (see \cite{biane_stochastic_1998}). This level of generality will not be needed in the sequel.
\end{remark}
 We will also make use of the pseudo-inverse of the number operator $N_0^{-1}$, whose action on a multiple Wigner integral of order $n \geq 1$ is given by $N_0^{-1} I_n(f) = \frac{1}{n}I_n(f)$.

 Before concluding this section, we introduce $\sharp$ to be the associative
 action of $\mathscr{A} \otimes \mathscr{A}^{\operatorname{op}}$ (where
 $\mathscr{A}^{\operatorname{op}}$ denotes the opposite algebra) on $\mathscr{A}
 \otimes \mathscr{A}$, as 
\begin{equation}
\label{sharpdef}
(A\otimes B) \sharp (C \otimes D) = (AC) \otimes (DB).
\end{equation}
Furthermore, we also write $\sharp$ to denote the action of $\mathscr{A} \otimes L^2\left( \mathbb{R}_{+}\right) \otimes  \mathscr{A}^{\operatorname{op}}$ on $\mathscr{A} \otimes L^2\left( \mathbb{R}_{+}\right) \otimes  \mathscr{A}$, as
\begin{equation*}
\label{sharpdefwithhilbert}
(A\otimes f \otimes  B) \sharp (C \otimes g \otimes  D) = (AC) \otimes fg \otimes  (DB).
\end{equation*}
 The multiplication $\sharp$ naturally appears in the following bound
 from~\cite{kemp_wigner_2012} on the $d_{\mathcal{C}_2}$ distance introduced
 above.
 
\begin{theorem}[\cite{kemp_wigner_2012}]
Let $S$ be a standard semicircular random variable and $F \in L^2(\mathcal{S},\varphi)$ be self-adjoint, in the domain of the free gradient $\nabla$ and such that $\varphi(F)=0$. Then, 
\begin{equation}
\label{kempbound}
d_{\mathcal{C}_2}(F,S) \leq \frac{1}{2}\varphi \otimes \varphi\left( \left | \int_{\mathbb{R}_{+}}\nabla_s\left( N_0^{-1} F\right) \sharp \left(\nabla_s F \right)^{*}  ds- 1\otimes 1 \right |\right).
\end{equation}
\end{theorem}

\section{Main results}
\label{mainressections}

\subsection{Bicontractions and biproduct formula}
As announced in the introduction, we will need an extension of the product
formula~\eqref{productformula} from~\cite{biane_stochastic_1998}. To this end,
we introduce the notion of bicontraction.

\begin{definition}
\label{defbicont}
Let $n_1, m_1, n_2, m_2$ be positive integers. Let $f \in L^2\left(
  \mathbb{R}_{+}^{n_1}\right) \otimes L^2\left( \mathbb{R}_{+}^{m_1}\right)
\cong L^2\left( \mathbb{R}_{+}^{n_1+m_1}\right)$ and $g \in L^2\left(
  \mathbb{R}_{+}^{n_2}\right) \otimes L^2\left( \mathbb{R}_{+}^{m_2}\right)
\cong L^2\left( \mathbb{R}_{+}^{n_2+m_2}\right)$ and let $p \leq n_1 \wedge
n_2$, $r \leq m_1 \wedge m_2$ be natural numbers. The $(p,r)$-bicontraction $f
\conts{p,r} g$ is the $L^2\left( \mathbb{R}_{+}^{n_1+n_2 -2p}\right) \otimes
L^2\left( \mathbb{R}_{+}^{m_1+m_2 -2r}\right) \cong L^2\left(
  \mathbb{R}_{+}^{n_1+n_2+m_1+m_2 -2p-2r}\right)$ function defined by 
\begin{multline*}
f \conts{p,r} g(  t_1, \ldots , t_{n_1+n_2+m_1+m_2 -2p-2r} ) 
\\ 
\begin{aligned}
 = \int_{\mathbb{R}_{+}^{p+r}}f(t_1, \ldots , t_{n_1-p},&s_p,\ldots,s_1, y_1,\ldots, y_r, \\ &  t_{n_1 + n_2 +m_2 -2p -r +1}, \ldots , t_{n_1 + n_2 + m_1 +m_2 -2p -2r}  ) 
\\  &\times g\left( s_1 , \ldots , s_p , t_{n_1-p+1},\ldots, t_{n_1+n_2+m_2 -2p-r}, y_r, \ldots, y_1\right)
\end{aligned}
\\ ds_1 \cdots ds_p dy_1 \cdots dy_r.
\end{multline*}
\end{definition}
\begin{remark}
Observe that for $f = f_1 \otimes f_2$ and $g = g_1 \otimes g_2$ with $f_1 \in L^2\left( \mathbb{R}_{+}^{n_1}\right)$, $f_2 \in L^2\left( \mathbb{R}_{+}^{m_1}\right)$, $g_1 \in L^2\left( \mathbb{R}_{+}^{n_2}\right)$ and $g_2 \in L^2\left( \mathbb{R}_{+}^{m_2}\right)$, the above definition reads
\begin{equation}
\label{separarblecontractioncase}
f \conts{p,r} g = \left( f_1 \otimes f_2\right) \conts{p,r} \left( g_1 \otimes g_2\right)  = \left( f_1 \cont{p}  g_1\right) \otimes \left( g_2 \cont{r}  f_2\right),
\end{equation}
where the contractions appearing on the right-hand side are the nested
contractions introduced in Definition~\ref{defcontractions}.
\end{remark}
\begin{remark}
In what follows, for $f,g$ as in Definition \ref{defbicont}, we write $f
\conts{p,r} g$ and $f \conts{s} g$ to denote the bicontraction and contraction
of $f$ and $g$, respectively. Here, we have somewhat abused notation by using
the same symbol for a function living in $L^2\left( \mathbb{R}_{+}^{n_1}\right)
\otimes L^2\left( \mathbb{R}_{+}^{m_1}\right)$ or its identification in
$L^2\left( \mathbb{R}_{+}^{n_1+m_1}\right)$. However, it will always be clear
from the type of contraction used which version of the function is being
considered.  
\end{remark}
 The following result collects some properties of bicontractions in the
 case where both functions are symmetric. 
\begin{lemma}
\label{propertiesofconts}
For $n_1, m_1, n_2, m_2 \in \mathbb{N}$, let $f \in L^2\left(
  \mathbb{R}_{+}^{n_1}\right) \otimes L^2\left( \mathbb{R}_{+}^{m_1}\right)
\cong L^2\left( \mathbb{R}_{+}^{n_1+m_1}\right)$ and $g \in L^2\left(
  \mathbb{R}_{+}^{n_2}\right) \otimes L^2\left( \mathbb{R}_{+}^{m_2}\right)
\cong L^2\left( \mathbb{R}_{+}^{n_2+m_2}\right)$ be fully symmetric
functions. Furthermore, let $p \leq n_1 \wedge n_2$ and $r \leq m_1 \wedge m_2$ be natural
numbers such that $p+r = p'+r'$. Then, the following is true.
\begin{enumerate}[(i)]
\item $f \conts{p,r} g \cong f \conts{p+r} g$.
\item $f \conts{p,r} g =f \conts{p',r'} g$.
\item $\norm{f \conts{p,r} g}_{L^2\left( \mathbb{R}_{+}^{n_1+n_2 -2p}\right) \otimes L^2\left( \mathbb{R}_{+}^{m_1+m_2 -2r}\right)}^{2} = \norm{f \cont{p+r} g}_{L^2\left( \mathbb{R}_{+}^{n_1+n_2+m_1+m_2 -2p-2r}\right)}^{2}$.
\item $f \conts{n_1,m_1} f = \norm{f}_{L^2\left( \mathbb{R}_{+}^{n_1}\right) \otimes L^2\left( \mathbb{R}_{+}^{m_1}\right)}^{2} 1\otimes 1$, which is a constant in $L^2\left( \mathbb{R}_{+}^{n_1}\right) \otimes L^2\left( \mathbb{R}_{+}^{m_1}\right)$.
\end{enumerate}
\end{lemma}
\begin{proof}
Just exploit the full symmetry of $f$ in the above definition of contractions.
\end{proof}
We are now ready to state the biproduct formula, which will be a crucial tool in
order to prove our main result.
\begin{theorem}
\label{productformulabi}
For $n_1, m_1, n_2, m_2 \in \mathbb{N}$, let $f \in L^2\left(
  \mathbb{R}_{+}^{n_1}\right) \otimes L^2\left( \mathbb{R}_{+}^{m_1}\right)
\cong L^2\left( \mathbb{R}_{+}^{n_1+m_1}\right)$ and $g \in L^2\left(
  \mathbb{R}_{+}^{n_2}\right) \otimes L^2\left( \mathbb{R}_{+}^{m_2}\right)
\cong L^2\left( \mathbb{R}_{+}^{n_2+m_2}\right)$. Then it holds that 
\begin{equation}
\label{biintegralmultformula}
I_{n_1} \otimes I_{m_1}\left(f\right) \sharp I_{n_2} \otimes I_{m_2}\left(g\right)= \sum_{p=0}^{n_1 \wedge n_2}\sum_{r=0}^{m_1 \wedge m_2}I_{n_1+n_2 -2p} \otimes I_{m_1+m_2 -2r}\left(f \conts{p,r}g\right).
\end{equation}
\end{theorem}
\begin{proof}
Using a density argument together with the bisometry property of Wigner
bi-integrals, it is enough to prove the claim for functions $f$ and $g$ of the
type $a \otimes b$ where $a \in L^2\left( \mathbb{R}_{+}^{n}\right)$ and $b \in
L^2\left( \mathbb{R}_{+}^{m}\right)$ as the subset of functions $$\left\lbrace a
  \otimes b \colon a \in L^2\left( \mathbb{R}_{+}^{n}\right),\ b \in L^2\left(
    \mathbb{R}_{+}^{m}\right) \right\rbrace $$ is dense in $L^2\left(
  \mathbb{R}_{+}^{n}\right) \otimes L^2\left( \mathbb{R}_{+}^{m}\right)$). Let
therefore 
$f = a\otimes b$ with $a \in L^2\left( \mathbb{R}_{+}^{n_1}\right)$, $b \in
L^2\left( \mathbb{R}_{+}^{m_1}\right)$ and $g = c\otimes d$ with $c \in
L^2\left( \mathbb{R}_{+}^{n_2}\right)$, $d \in L^2\left(
  \mathbb{R}_{+}^{m_2}\right)$. It holds that 
\begin{align*}
I_{n_1} \otimes I_{m_1}\left(a \otimes b\right)\sharp I_{n_2} \otimes I_{m_2}\left(c\otimes d\right)
&= I_{n_1}\left( a\right)  \otimes I_{m_1}\left( b\right) \sharp I_{n_2}\left(c \right)  \otimes I_{m_2}\left( d\right)
\\ &= I_{n_1}\left( a\right) \cdot I_{n_2}\left( c\right) \otimes  I_{m_1}\left( d\right) \cdot I_{m_2}\left( b\right).
\end{align*}
Using the usual multiplication formula for Wigner integrals on both sides of the tensor product, we get
\begin{align*}
I_{n_1}\left( a\right) \cdot I_{n_2}\left( c\right) &\otimes
  I_{m_1}\left( d\right) \cdot I_{m_2}\left( b\right) \\ &= \left(
                                                                  \sum_{p=0}^{n_1
                                                                  \wedge
                                                                  n_2}I_{n_1+n_2
                                                                  -2p}\left(
                                                                  a \cont{p}
                                                                  c\right)
                                                                  \right)
                                                                  \otimes \left(
                                                                  \sum_{r=0}^{m_1
                                                                  \wedge
                                                                  m_2}I_{m_1+m_2
                                                                  -2r}\left(d
                                                                  \cont{r}
                                                                  b\right)\right)
\\ &= \sum_{p=0}^{n_1 \wedge n_2}\sum_{r=0}^{m_1 \wedge m_2}I_{n_1+n_2 -2p}
  \otimes I_{m_1+m_2 -2r}\left(\left( a \cont{p} c \right) \otimes \left( d
  \cont{r} b\right) \right)
\\ & =  \sum_{p=0}^{n_1 \wedge n_2}\sum_{r=0}^{m_1 \wedge m_2}I_{n_1+n_2 -2p}
      \otimes I_{m_1+m_2 -2r}\left(\left( a \otimes b \right) \conts{p,r}
      \left( c \otimes d\right) \right), 
\end{align*}
where the last equality follows from the identity
\eqref{separarblecontractioncase}. 
\end{proof}
\begin{remark}\hfill
\begin{enumerate}[1.]
\item By taking $m_1 = m_2 =0$, $f = u\otimes 1$ and $g = v \otimes 1$, we
  recover the usual product formula~\eqref{productformula} for Wigner integrals. 
\item Note that a similar version of the above biproduct formula also holds for
  the usual tensor product (with a slightly different definition for the
  bicontractions). Furthermore, using the same methodology, one could also
  define contractions and product formulae for higher order tensors.
\end{enumerate}
\end{remark}

\subsection{Quantitative Fourth Moment Theorems}
 We are now in the position of stating the main result of this paper, namely a bound on the quantity appearing in the right hand side of \eqref{kempbound} in terms of the fourth moment, which then leads to a quantitative Fourth Moment Theorem for multiple Wigner integrals.

\begin{theorem}
\label{maintheorem1}
For $n \in \mathbb{N}$, let $F = I_n\left(f \right) $ be a Wigner integral of
order $n$ with $f \in L^2\left(\mathbb{R}_{+}^{n} \right) $ symmetric and such
that $\norm{f}_{L^2\left(\mathbb{R}_{+}^{n} \right) }^{2}=1$. Then, it holds that
\begin{equation}
\label{mainbound1}
\varphi \otimes \varphi\left( \left | \int_{\mathbb{R}_{+}}\nabla_s\left(
      N_0^{-1} F\right) \sharp \left(\nabla_s F \right)^{*}  ds- 1\otimes 1
  \right |^2\right) \leq C_n \Big(\varphi\left( F^4\right) -2\Big), 
\end{equation}
where $C_n = \frac{1}{n^2}\max\left\lbrace P_{n}\left( \lfloor u_0
    \rfloor\right), P_{n}\left( \lceil u_0 \rceil\right)  \right\rbrace $ with 
\begin{align}
\label{expressionofu0}
P_{n}(u) &= \frac{1}{3}u^2(n-u+1)\left(2(n-u)^2 + 4(n-u) + 3 \right),\notag \\
u_0 &= \frac{1}{5}\left(4 (n+1) -\frac{r(n)}{\sqrt[3]{4}} -\frac{2 n^2+ 4 n -3}{ \sqrt[3]{2} r(n)}\right) \\
\intertext{and} r(n) &= \sqrt[3]{4 n^3+12 n^2+5 \sqrt{2} \sqrt{4 n^4+16 n^3+20 n^2+8 n+5}+22 n+14} \notag.
\end{align}
\end{theorem}
\begin{proof}
In the following we will use the shorthand $f^{(k)}_s$ to denote the function
given by
\begin{equation*}
  f_s^{(k)}(x_1,\dots,x_{n-1}) =
f(x_1,\dots,x_{k-1},s,x_{k+1},\dots,x_n).
\end{equation*}
 Observe that 
\begin{align*}
\int_{\mathbb{R}_{+}}\left( \nabla_s F\right)  \sharp \left(\nabla_s F
  \right)^{*}  &ds 
  \\ &= \sum_{k,q=1}^{n}\int_{\mathbb{R}_{+}} I_{k-1}\otimes
                      I_{n-k}\left(f_{s}^{(k)} \right) \sharp \left(
                      I_{q-1}\otimes I_{n-q}\left(f_{s}^{(q)}
                      \right)\right)^{*} ds \\ 
\\ &= \sum_{k,q=1}^{n}\int_{\mathbb{R}_{+}} I_{k-1}\otimes
    I_{n-k}\left(f_{s}^{(k)} \right) \sharp I_{q-1}\otimes
    I_{n-q}\left(f_{s}^{(q)} \right) ds, 
\end{align*}
where the last equality follows from the full symmetry of the function
$f$. Using the product formula for bi-integrals proven in Theorem
\ref{productformulabi} yields 
\begin{multline*}
\int_{\mathbb{R}_{+}}\left( \nabla_s F\right)  \sharp \left(\nabla_s F \right)^{*}  ds \\ = \sum_{k,q=1}^{n}\int_{\mathbb{R}_{+}}\sum_{p=0}^{\left( k \wedge q\right)  -1}\sum_{r=0}^{n - \left( k \vee q\right)  }I_{k+q -2p-2} \otimes I_{2n -k-q -2r}\left(f_s^{(k)} \conts{p,r} f_s^{(q)}\right)ds,
\end{multline*}
 and by a Fubini argument one gets
\begin{multline*}
\int_{\mathbb{R}_{+}}\left( \nabla_s F\right)  \sharp \left(\nabla_s F \right)^{*}  ds \\= \sum_{k,q=1}^{n}\sum_{p=0}^{\left( k \wedge q\right)  -1}\sum_{r=0}^{n - \left( k \vee q\right)  }I_{k+q -2p-2} \otimes I_{2n -k-q -2r}\left(\int_{\mathbb{R}_{+}} f_s^{(k)} \conts{p,r} f_s^{(q)} ds\right) .
\end{multline*}
The full symmetry of $f$ implies that $f_s^{(k)} = f_s^{(q)}$ for any $1 \leq k,q \leq n$, which together with Lemma \ref{propertiesofconts} yields $\int_{\mathbb{R}_{+}} f_s^{(k)} \conts{p,r} f_s^{(q)} ds = f \conts{p+r+1} f$. Hence,
\begin{equation}
\label{hilbertnormofmallderivative}
\int_{\mathbb{R}_{+}}\left( \nabla_s F\right)  \sharp \left(\nabla_s F \right)^{*}  ds = \sum_{k,q=1}^{n}\sum_{p=0}^{\left( k \wedge q\right)  -1}\sum_{r=0}^{n - \left( k \vee q\right)  }I_{k+q -2p-2} \otimes I_{2n -k-q -2r}\left(f \conts{p+r+1} f \right).
\end{equation}
Exactly those summands $I_{k+q -2p-2} \otimes I_{2n -k-q -2r}\left(f
  \conts{p+r+1} f \right)$ for which $k+q -2p-2 =0$ and $2n -k-q -2r =0$ yield
the constant term $\norm{f}_{L^2 \left( \mathbb{R}_+^n \right)}^{2} \cdot 1 \otimes 1$
(i.e. a constant in $L^2 \left( \mathbb{R}_+ \right)\otimes L^2 \left( \mathbb{R}_+ \right)$). These conditions,
along with the ranges of summation, imply that $k=q$ and $p+r+1 = n$. Therefore,
fixing $k$, for which we have $n$ possibilities, fixes the other three indices
$q$,$p$ and $r$ to take the values $k$, $k-1$ and $n-k$, respectively. Recalling
that $\norm{f}_{L^2 \left( \mathbb{R}_+^n \right)}^{2} = 1$,
\eqref{hilbertnormofmallderivative} can thus be rewritten as
\begin{multline*}
\int_{\mathbb{R}_{+}}\left( \nabla_s F\right)  \sharp \left(\nabla_s F
  \right)^{*}  ds \\ = n \cdot 1 \otimes 1  + \sum_{k,q=1}^{n}\sum_{p=0}^{\left( k
                      \wedge q\right) -1 }\sum_{r=0}^{n - \left( k \vee q\right)
                      }\mathds{1}_{\left\lbrace n-1-p-r >0\right\rbrace }
                      \\ \times
                      I_{k+q -2p-2} \otimes I_{2n -k-q -2r}\left(f
                      \conts{p+r+1} f \right),
\end{multline*}
which, by using that $N_{0}^{-1}F = \frac{1}{n}F$, gives
\begin{align}
  \notag
  \int_{\mathbb{R}_{+}}\nabla_s&\left( N_0^{-1} F\right) \sharp \left(\nabla_s F
                               \right)^{*}  ds  - 1 \otimes 1
  \\ &= \frac{1}{n}\int_{\mathbb{R}_{+}}\left(
                                     \nabla_s F\right)  \sharp \left(\nabla_s F
       \right)^{*}  ds - 1 \otimes 1 \nonumber
  \\ &= \notag
       \frac{1}{n}\sum_{k,q=1}^{n}\sum_{p=0}^{\left( k \wedge q\right) -1
       }\sum_{r=0}^{n - \left( k \vee q\right)  }\mathds{1}_{\left\lbrace
       n-1-p-r >0\right\rbrace } 
  \\ & \qquad \qquad \qquad \qquad \qquad  I_{k+q -2p-2} \otimes I_{2n -k-q
       -2r}\left(f \conts{p+r+1} f \right) \nonumber
  \\ &=
       \label{combinatoriallyconvieninentexpression}
       \frac{1}{n}\sum_{p,r=0}^{n-1}\sum_{k,q=0}^{n-1-p-r}\mathds{1}_{\left\lbrace
       n-1-p-r >0\right\rbrace }\cdot I_{k+q} \otimes
       I_{2(n-1-p-r)-k-q}\left(f \conts{p+r+1} f \right). 
\end{align}
Grouping all occuring bi-integrals by the order of the contraction, one arrives
at 
\begin{multline}
\label{expressionwiththeconstants}
\int_{\mathbb{R}_{+}}\nabla_s\left( N_0^{-1} F\right) \sharp \left(\nabla_s F
\right)^{*}  ds  - 1 \otimes 1 
\\ = \frac{1}{n}\sum_{u=1}^{n-1}\sum_{v=0}^{2(n-u)}
c_{u,v} I_{v} \otimes I_{2(n-u)-v}\left(f \conts{u} f \right), 
\end{multline}
where the $c_{u,v}$ are positive constants depending solely on $u$ and $v$.
Taking the trace of the square of \eqref{expressionwiththeconstants} and using
the Wigner bisometry \eqref{wignerbisometry} yields
\begin{align*}
\varphi\otimes\varphi\left(\left | \int_{\mathbb{R}_{+}}\nabla_s\left( N_0^{-1}
  F\right)
  \right. \right. &\sharp \left. \left. \left(\nabla_s F \right)^{*}  ds - 1 \otimes 1 \right |^{2}
  \right)
  \\  &= \frac{1}{n^2}\sum_{u=1}^{n-1}\sum_{v=0}^{2(n-u)} c_{u,v}^2 \norm{f
            \conts{u} f}_{L^2 \left( \mathbb{R}_+^{2n-2u} \right)}^{2}
  \\ &\leq 
         \frac{1 }{n^2}\max_{1 \leq u \leq n-1}\left\lbrace  \sum_{v=0}^{2(n-u)}
         c_{u,v}^2 \right\rbrace  \sum_{u=1}^{n-1} \norm{f \conts{u}
         f}_{L^2 \left( \mathbb{R}_+^{2n-2u} \right)}^{2}. 
\end{align*}
As is well known, 
\begin{equation*}
\sum_{u=1}^{n-1} \norm{f \conts{u} f}_{L^2 \left( \mathbb{R}_+^{2n-2u} \right)}^{2} = \varphi\left( F^4\right) -2,
\end{equation*}
so that it only remains to evaluate the maximum. To this end, the constants
$c_{u,v}$ will be computed explicitly. By carefully comparing
\eqref{combinatoriallyconvieninentexpression}
with~\eqref{expressionwiththeconstants}, one sees that $c_{u,v}$ is
given by the cardinality of the set of all quadruples $(p,r,k,q)$ satisfying the
following conditions:
\begin{align*}
&0 \leq p,r \leq n-1; &
&0 \leq k,q \leq n-1-p-r;\\
&k+q  = v;&
&p+r = u-1;\\
&p+r <n-1.&&
\end{align*}
By reindexing the second sum in~\eqref{combinatoriallyconvieninentexpression}
via the transformation $(k',q') = (n-1-p-r-k,n-1-p-r-q)$, we see that $c_{u,v} =
c_{u, 2(n-u)-v}$, thus only the constants $c_{u,v}$ for which $v \leq n-u$ need
to be computed explicitly. Fix $u$ and $v$. Then, there are $u$ couples $(p,r)$ 
satisfying $p+r = u-1$, namely $(0, u-1), (1,u-2), \ldots , (u-1,0)$. Likewise, there
are $v+1$ couples $(k,q)$ satisfying $k+q = v$. Therefore,
\begin{equation*}
c_{u,v} = \begin{cases} u(v+1) & \quad \text{if $v \leq n-u$} \\ u(2(n-u) - v
  +1) & \quad \text{if $v > n-u$}. \end{cases} 
\end{equation*}
This yields
\begin{align*}
\sum_{v=0}^{2(n-u)} c_{u,v}^2 &= \sum_{v=0}^{n-u} u^2(v+1)^2 +
                                \sum_{v=n-u+1}^{2(n-u)} u^2(2(n-u) - v +1)^2
  \\ &=
       \sum_{v=0}^{n-u} u^2(v+1)^2 + \sum_{v=0}^{n-u-1} u^2(v +1)^2
  \\ &=
       \frac{1}{3}u^2\left( n-u\right)(n-u+1)(2(n-u)+1) + u^2(n-u+1)^2
  \\ &=
       P_{n}(u).
\end{align*}
Straightforward analysis shows that the polynomial $P_{n}$ has exactly one
maximum in the interval $\left(  1,n-1 \right)$ attained at $u_0$ as defined in
\eqref{expressionofu0}. Therefore, to maximize $P_n$, one has to select the
closest integer to $u_0$.
\end{proof}
 Combining Theorem \ref{maintheorem1} with the bound appearing in
 \eqref{kempbound} and applying the Cauchy-Schwarz inequality immediately yields
 the following quantitative free Fourth Moment Theorem.
\begin{corollary}
\label{propquant1}
Let $n \geq 2$ be a natural number and $F = I_n\left(f \right) $, where $f$ is a
symmetric function in $L^2\left(\mathbb{R}_{+}^{n} \right) $ such that
$\norm{f}_{L^2\left(\mathbb{R}_{+}^{n} \right) }^{2} =1$. Let $S$ be standard
semicircular random variable. Then it holds that, 
\begin{equation*}
d_{\mathcal{C}_2}(F,S) \leq \frac{\sqrt{C_n}}{2}\sqrt{\varphi\left( F^4\right) -2},
\end{equation*}
where $C_n$ is the constant appearing in Theorem \ref{maintheorem1}.
\end{corollary}
\begin{remark}\hfill
\label{bigremark}
\begin{enumerate}[1.]
\item It holds that $C_2 = \frac{3}{2}$, so that Corollary \ref{propquant1} becomes
\begin{align*}
d_{\mathcal{C}_2}\left( I_{2}(f),S\right) & \leq \frac{1}{2}\sqrt{\frac{3}{2}}\sqrt{\varphi\left( I_{2}(f)^4\right) -2},
\intertext{which is precisely the conclusion of \cite[Corollary 1.12]{kemp_wigner_2012}. The next few values of $C_n$ are given by $C_3 = 2$, $C_4 = \frac{19}{4}$ yielding}
d_{\mathcal{C}_2}\left( I_{3}(f),S\right) & \leq \frac{1}{\sqrt{2}}\sqrt{\varphi\left( I_{3}(f)^4\right) -2}, \\
\intertext{and}
d_{\mathcal{C}_2}\left( I_{4}(f),S\right) & \leq \frac{\sqrt{19}}{4}\sqrt{\varphi\left( I_{4}(f)^4\right) -2},
\end{align*}
where $f$ is of course always chosen appropriately to be symmetric and an element of $L^2\left( \mathbb{R}_{+}^{n}\right)$ for each $n=2,3,4$. 
\item In general, a straightforward analysis shows that $C_n$ grows with $n$. In
  the commutative case, when bounding the distance between a multiple Wiener
  integral of any order and the standard Gaussian distribution by means of the
  fourth moment, the constants appearing in the bounds do not depend on the
  order of the multiple integral (see for
  example~\cite{nourdin_steins_2009}). If such a dimension-free bound also holds
  in the free case is not known and this question is left for future
  research.
\item As stated in the introduction, convergence of the fourth moment to $2$ also
  implies convergence of multiple integral with mirror-symmetric kernels
  towards the semicircular distribution. As one needs the function $f$ to be
  symmetric in Theorem \ref{maintheorem1}, it is natural to ask if the
  bound~\eqref{mainbound1} also holds for mirror-symmetric kernels. As the
  following counterexample shows, this is not true. Divide $\left[0,1 \right]$
  into $N$ 
  intervals $I_1,I_2,\ldots,I_N$ of equal length $\frac{1}{N}$ and define the
  function $$f_N(x_1,x_2,x_3) = \sqrt{N}\sum_{k=1}^{N}\mathds{1}_{I_k \times
    I_k}(x_1,x_3)$$ on $\left[ 0,1\right]^3$. Observe that $f_N$ is a
  mirror-symmetric function in $L^2\left( \left[0,1 \right]^3\right) $. Then,  
\begin{equation*}
I_{3}\left( f_N\right) = \sqrt{N}\sum_{k=1}^{N}I_1\left( \mathds{1}_{I_k}\right) I_1\left( 1\right) I_1\left( \mathds{1}_{I_k}\right).
\end{equation*}
It is easy to check that $\varphi\left( I_{3}(f_N)^2\right) = 1$ and
\begin{equation*}
\varphi\left( I_{3}\left( f_N\right)^4\right) -2= \norm{f_N \conts{1} f_N}_{L^2\left( \left[0,1 \right]^4\right)}^{2} +\norm{f_N \conts{2} f_N}_{L^2\left( \left[0,1 \right]^2\right)}^{2} = \frac{2}{N},
\end{equation*}
implying (by the free Fourth Moment Theorem of~\cite{kemp_wigner_2012}) that the
sequence $\left\lbrace I_{3}\left( f_N\right) \colon N \geq 1\right\rbrace $
converges in distribution to the standard semicircular law. Furthermore,
\begin{multline*}
\nabla_t I_{3}\left( f_N\right)  = \sqrt{N}\sum_{k=1}^{N}[
  \mathds{1}_{I_k}(t) \otimes I_1\left( 1\right) I_1\left(
    \mathds{1}_{I_k}\right) \\ +  I_1\left( \mathds{1}_{I_k}\right) \otimes
  I_1\left( \mathds{1}_{I_k}\right) +  I_1\left(
    \mathds{1}_{I_k}\right)I_1\left( 1\right) \otimes
  \mathds{1}_{I_k}(t)].  
\end{multline*}
As $f_N$ is a sum of products of non-negative indicator functions, the quantity 
\begin{equation}
\label{quantitylargerthan1}
\varphi \otimes \varphi\left( \left | \int_{\mathbb{R}_{+}}\nabla_s\left( N_0^{-1} I_{3}\left( f_N\right) \right) \sharp \left(\nabla_s I_{3}\left( f_N\right) \right)^{*}  ds- 1\otimes 1 \right |^2\right)
\end{equation}  
is a sum of non-negative terms. Hence, if one of these terms can be proven not
to converge to zero, the entire quantity must be bounded away from zero as
well. One of the summands appearing is $\sqrt{N}\sum_{k=1}^{N}  I_1\left(
  \mathds{1}_{I_k}\right) \otimes I_1\left( \mathds{1}_{I_k}\right)$. It holds
that  
\begin{multline*}
\left(\sqrt{N}\sum_{k=1}^{N}  I_1\left( \mathds{1}_{I_k}\right) \otimes
  I_1\left( \mathds{1}_{I_k}\right) \right) \sharp \left(\sqrt{N}\sum_{k=1}^{N}
  I_1\left( \mathds{1}_{I_k}\right) \otimes I_1\left( \mathds{1}_{I_k}\right)
\right)^{*}
\\ = N\sum_{k,q=1}^{N}I_1\left( \mathds{1}_{I_k}\right)I_1\left(
  \mathds{1}_{I_q}\right)\otimes I_1\left( \mathds{1}_{I_q}\right)I_1\left(
  \mathds{1}_{I_k}\right)
\end{multline*}
and a straightforward calculation shows that
\begin{equation*}
\varphi \otimes \varphi\left( \left| N\sum_{k,q=1}^{N}I_1\left( \mathds{1}_{I_k}\right)I_1\left( \mathds{1}_{I_q}\right)\otimes I_1\left( \mathds{1}_{I_q}\right)I_1\left( \mathds{1}_{I_k}\right)\right|^2\right) = 1 +\frac{3}{N},
\end{equation*}
which does not go to zero as $N$ goes to infinity. In total, it holds that
$\varphi\left( I_{3}\left( f_N\right)^4\right) -2 \rightarrow 0$ as $N$ goes to
infinity, but the quantity \eqref{quantitylargerthan1} is strictly greater than
1 for all $N$, hence proving that the quantity \eqref{quantitylargerthan1} can
not be controlled by the fourth moment. Therefore, Theorem \ref{maintheorem1}
can not be extended to mirror-symmetric kernels.
\end{enumerate}
\end{remark}
 In the commutative case, the classical Nualart-Ortiz-Latorre equivalence
 criterion for normal convergence of multiple Wiener integrals $F_k =
 I_{n}^{W}\left( f_k\right) $ reads  
\begin{equation*}
\int_{\mathbb{R}_{+}}\left( D_sF_k\right) ^2ds \rightarrow n \ \text{ in }\
L^2\left(\Omega\right),\ \text{as $k \rightarrow \infty$}, 
\end{equation*}
where $D$ denotes the Malliavin gradient and $\Omega$ stands for the underlying
probability space (see~\cite{nualart_central_2008}). For Wigner integrals, an
analogue of this criterion was only known to hold in the second chaos (see
\cite[Theorem 4.8]{kemp_wigner_2012}). Theorem \ref{maintheorem1} extends this
to any order of chaos. Therefore, all equivalent criteria for normal convergence
of Wiener integrals have now free analogues for convergence of Wigner integrals
towards a semicircular distribution. For the sake of completeness, we collect
these analogues in the following Theorem.
\begin{theorem}
\label{nualartortizlatorrefree}
Let $n \geq 2$ be a natural number and let $\left\lbrace f_k \colon k \geq 1
\right\rbrace $ be a sequences of symmetric functions in $L^2\left(
  \mathbb{R}_{+}^{n}\right) $ such that, for all $k \geq 1$,
$\norm{f_k}_{L^2\left( \mathbb{R}_{+}^{n}\right)} = 1$. For any $k \geq 1$,
denote $F_k = I_{n}\left(f_k\right)$. Then, the following conditions are
equivalent: 
\begin{enumerate}[(i)]
\item The sequence $\left\lbrace F_k \colon k \geq 1  \right\rbrace $ converges
  in law to the standard semicircular distribution.
\item As $k$ tends to infinity, $\varphi\left(F_k ^4 \right) \rightarrow 2$.
\item For all $1 \leq p \leq n-1$, as $k$ tends to infinity, $\norm{f_k \conts{p} f_k}_{L^2\left(\mathbb{R}_{+}^{2n-2p} \right) } \rightarrow 0$.
\item As $k$ tends to infinity, 
\begin{equation*}
\int_{\mathbb{R}_{+}}\left( \nabla_s F_k\right)  \sharp \left(\nabla_s F_k \right)^{*}ds \rightarrow n \cdot 1 \otimes 1\ \text{ in }\  L^2\left(\mathcal{S}\otimes\mathcal{S}, \varphi \otimes \varphi \right).
\end{equation*}
\end{enumerate}
\end{theorem}
\begin{proof}
The equivalences $(i)\Leftrightarrow (ii) \Leftrightarrow (iii)$ and the
implication $(iv) \Rightarrow (i)$ follow from \cite[Theorem 1.6, Theorem
1.10]{kemp_wigner_2012}. The missing implication $(ii) \Rightarrow (iv)$ follows
from the main result of this section, namely Theorem \ref{maintheorem1}.
\end{proof}

\section{Quantifying the free Breuer-Major theorem}
\label{breuermajorsection}
 Our main results can be used to provide Berry-Esseen bounds for a free version
 of the Breuer-Major theorem (see \cite{kemp_wigner_2012}) for the free
 fractional Brownian motion. This can be regarded as a free analog of
 \cite[Theorem 4.1]{nourdin_steins_2009}. The free fractional
 Brownian motion $S^{H}$ with index $H \in \left(-1,1 \right) $ is defined as a
 centered semicircular process with covariance function
\begin{equation*}
\varphi\left(S_{t}^{H}S_{s}^{H}\right) = \frac{1}{2}\left( t^{2H} + s^{2H} - \vert t-s \vert^{2H}\right).
\end{equation*}
 As is well-known (see for example \cite{biane_stochastic_1998} or
 \cite{nica_lectures_2006}), the orthogonal polynomials associated to the
 semicircular distribution are the Chebyshev polynomials $U_n$ of the second kind
 defined on $\left[-2,2 \right]$ by the recurrence relations $U_{0}(x)= 1$,
 $U_{1}(x) = x$, and for $n \geq 2$,
\begin{equation*}
U_{n+1}(x) = xU_{n}(x)- U_{n-1}(x).
\end{equation*}
 For $n\in \mathbb{N}$, define the increment sequence $\left\lbrace X_k =
   S_{k+1}^{H} - S_{k}^{H} \colon k \geq 0\right\rbrace $. Straightforward
 calculations show that the autocovariance function $\rho_H(k)$ is given by
\begin{equation*}
\rho_{H}(k) = \varphi(X_0 X_{k}) = \frac{1}{2}\left( \vert k+1
  \vert^{2H} + \vert k-1 \vert^{2H} - 2\vert k \vert^{2H}\right).
\end{equation*}  
Furthermore, define $\left\lbrace V_m \colon m\geq 1\right\rbrace $ as
\begin{equation*}
V_m = \frac{1}{\sqrt{m}}\sum_{k=0}^{m-1}U_n\left(X_k \right).
\end{equation*}
With these definitions in place, we can now state the announced Berry-Esseen bounds.
\begin{theorem}
With the above notation prevailing, suppose that there exists an integer $n \geq
1$ such that $\sigma^2 = \sum_{k \in \mathbb{Z}} \left | \rho_H(k) \right | ^n <
\infty$. Then, there exists a positive constant $C_{n,H}$ such that
\begin{equation*}
d_{\mathcal{C}_2}\left(\frac{V_m}{\sigma}, \mathcal{S}(0,1) \right)  \leq
C_{n,H}m^{\alpha\left(n,H \right) }, 
\end{equation*}
where the function $\alpha\left( n,H\right) $ is given by 
\begin{equation*}
\alpha\left(n,H \right) = \begin{cases} m^{-\frac{1}{2}} & \text{if}\ H \in
  \left(0, \frac{1}{2} \right], \\ m^{H-1} & \text{if}\ H \in \left[\frac{1}{2},
    \frac{2n-3}{2n-2} \right], \\ m^{nH-n+\frac{1}{2}} & \text{if}\ H \in
  \left[\frac{2n-3}{2n-2},  \frac{2n-1}{2n}\right). \end{cases} 
\end{equation*}
\end{theorem}
\begin{proof}
It is well known (see e.g. \cite[Proposition 2.5]{nourdin_selected_2012}), that
the (non-free) fractional Brownian motion can be represented as a Wiener
integral with respect to a standard Brownian motion as
\begin{equation*}
B^{H}_t = \int_{0}^{t}K_H\left(t,u \right)dW_u,
\end{equation*}
where the kernel $K_H(\cdot, \cdot)$ is explicit (see e.g. \cite[Proposition
2.5]{nourdin_selected_2012}). Using the correspondence between Wiener and Wigner
integrals, it also holds that  
\begin{equation*}
S^{H}_t = \int_{0}^{t}K_H\left(t,u \right)dS_u.
\end{equation*}
Indeed, this can be verified by checking that the covariance function of the
above integral coincides with the one of the free fractional Brownian
motion. Then, denoting
\begin{equation*}
f_{k,m,H} =  m^H \left( K_H\left(\frac{k+1}{m} , \cdot \right)\mathds{1}_{\left[
      0,\frac{k+1}{m}\right] } - K_H\left(\frac{k}{m} , \cdot
  \right)\mathds{1}_{\left[ 0,\frac{k}{m}\right] }\right), 
\end{equation*}
it holds that
\begin{equation*}
V_m = \frac{1}{\sqrt{m}}\sum_{k=0}^{m-1}U_n\left(I_{1}\left(f_{k,m,H} \right)  \right)
\end{equation*}
Observe that $\norm{f_{k,m,H}}_{L^2\left( \mathbb{R}_{+}\right) } = 1$ so that
\begin{equation*}
V_m = \frac{1}{\sqrt{m}}\sum_{k=0}^{m-1}I_{n}\left(f_{k,m,H}^{\otimes n}\right)
= I_{n}\left(\frac{1}{\sqrt{m}}\sum_{k=0}^{m-1} f_{k,m,H}^{\otimes n}\right). 
\end{equation*}
Define
\begin{equation*}
g_{n,m,H} = \frac{1}{\sigma\sqrt{m}}\sum_{k=0}^{m-1} f_{k,m,H}^{\otimes n}.
\end{equation*}
Applying Corollary~\ref{propquant1} to $V_m$, we get 
\begin{align*}
d_{\mathcal{C}_2}\left( \frac{V_m}{\sigma},\mathcal{S}(0,1)\right) 
&\leq
\frac{\sqrt{C_n}}{2}\sqrt{\varphi\left( I_n\left(g_{n,m,H}\right)^4 \right) -2}
\\ &= \frac{\sqrt{C_n}}{2}\sqrt{\sum_{u=1}^{n-1}\norm{g_{n,m,H}  \conts{u} g_{n,m,H}
  }_{L^2\left( \R_{+}^{2n-2u}\right) }^{2}}. 
\end{align*} 
From here, one can evaluate and estimate the contraction norms similarly as in
the proof of~\cite[Theorem 4.1]{nourdin_steins_2009}. 
\end{proof}

\begin{acknow*}
The authors wish to thank Roland Speicher for providing the counterexample
appearing in Remark \ref{bigremark} and Tobias Mai and Roland Speicher for
several stimulating discussions. S. Campese was partially supported by ERC grant
277742 Pascal.  
\end{acknow*}

\bibliographystyle{alpha}
\bibliography{biblio}

\begin{thebibliography}{CNPP16}

\bibitem[ACP14]{azmoodeh_fourth_2014}
Ehsan Azmoodeh, Simon Campese, and Guillaume Poly.
\newblock Fourth {Moment} {Theorems} for {Markov} diffusion generators.
\newblock {\em Journal of Functional Analysis}, 266(4):2341--2359, 2014.

\bibitem[Bou15]{bourguin_poisson_2015}
Solesne Bourguin.
\newblock Poisson convergence on the free {Poisson} algebra.
\newblock {\em Bernoulli. Official Journal of the Bernoulli Society for
  Mathematical Statistics and Probability}, 21(4):2139--2156, 2015.

\bibitem[Bou16]{bourguin_vector-valued_2016}
Solesne Bourguin.
\newblock Vector-valued semicircular limits on the free {Poisson} chaos.
\newblock {\em Electronic Communications in Probability}, 21(55):1--11, 2016.

\bibitem[BP14a]{bourguin_portmanteau_2014}
Solesne Bourguin and Giovanni Peccati.
\newblock Portmanteau inequalities on the {Poisson} space: mixed regimes and
  multidimensional clustering.
\newblock {\em Electronic Journal of Probability}, 19:no. 66, 42, 2014.

\bibitem[BP14b]{bourguin_semicircular_2014}
Solesne Bourguin and Giovanni Peccati.
\newblock Semicircular limits on the free {Poisson} chaos: {Counterexamples} to
  a transfer principle.
\newblock {\em Journal of Functional Analysis}, 267(4):963--997, August 2014.

\bibitem[BS98]{biane_stochastic_1998}
Philippe Biane and Roland Speicher.
\newblock Stochastic calculus with respect to free {Brownian} motion and
  analysis on {Wigner} space.
\newblock {\em Probability Theory and Related Fields}, 112(3):373--409, 1998.

\bibitem[CNPP16]{campese_multivariate_2016}
Simon Campese, Ivan Nourdin, Giovanni Peccati, and Guillaume Poly.
\newblock Multivariate {Gaussian} approximations on {Markov} chaoses.
\newblock {\em Electronic Communications in Probability}, 21, 2016.

\bibitem[DN12]{deya_convergence_2012}
Aur\'{e}lien Deya and Ivan Nourdin.
\newblock Convergence of {Wigner} integrals to the tetilla law.
\newblock {\em ALEA. Latin American Journal of Probability and Mathematical
  Statistics}, 9:101--127, 2012.

\bibitem[DNN13]{deya_fourth_2013}
Aur\'{e}lien Deya, Salim Noreddine, and Ivan Nourdin.
\newblock Fourth moment theorem and $q$-{Brownian} chaos.
\newblock {\em Communications in Mathematical Physics}, 321(1):113--134, 2013.

\bibitem[HP00]{hiai_semicircle_2000}
Fumio Hiai and D\'{e}nes Petz.
\newblock {\em The semicircle law, free random variables and entropy},
  volume~77 of {\em Mathematical {Surveys} and {Monographs}}.
\newblock American Mathematical Society, Providence, RI, 2000.

\bibitem[KNPS12]{kemp_wigner_2012}
Todd Kemp, Ivan Nourdin, Giovanni Peccati, and Roland Speicher.
\newblock Wigner chaos and the fourth moment.
\newblock {\em The Annals of Probability}, 40(4):1577--1635, 2012.

\bibitem[Led12]{ledoux_chaos_2012}
M.~Ledoux.
\newblock Chaos of a {Markov} operator and the fourth moment condition.
\newblock {\em The Annals of Probability}, 40(6):2439--2459, November 2012.
\newblock Zentralblatt MATH identifier: 06114704.

\bibitem[NOL08]{nualart_central_2008}
D.~Nualart and S.~Ortiz-Latorre.
\newblock Central limit theorems for multiple stochastic integrals and
  {Malliavin} calculus.
\newblock {\em Stochastic Processes and their Applications}, 118(4):614--628,
  2008.

\bibitem[Nou12]{nourdin_selected_2012}
Ivan Nourdin.
\newblock {\em Selected aspects of fractional {Brownian} motion}, volume~4 of
  {\em Bocconi \& {Springer} {Series}}.
\newblock Springer, Milan, 2012.

\bibitem[NP05]{nualart_central_2005}
David Nualart and Giovanni Peccati.
\newblock Central limit theorems for sequences of multiple stochastic
  integrals.
\newblock {\em The Annals of Probability}, 33(1):177--193, 2005.

\bibitem[NP09a]{nourdin_noncentral_2009}
Ivan Nourdin and Giovanni Peccati.
\newblock Noncentral convergence of multiple integrals.
\newblock {\em The Annals of Probability}, 37(4):1412--1426, July 2009.

\bibitem[NP09b]{nourdin_steins_2009}
Ivan Nourdin and Giovanni Peccati.
\newblock Stein's method on {Wiener} chaos.
\newblock {\em Probability Theory and Related Fields}, 145(1-2):75--118, 2009.

\bibitem[NP13]{nourdin_poisson_2013}
Ivan Nourdin and Giovanni Peccati.
\newblock Poisson approximations on the free {Wigner} chaos.
\newblock {\em The Annals of Probability}, 41(4):2709--2723, 2013.

\bibitem[NPR10]{nourdin_multivariate_2010}
Ivan Nourdin, Giovanni Peccati, and Anthony R\'{e}veillac.
\newblock Multivariate normal approximation using {Stein}'s method and
  {Malliavin} calculus.
\newblock {\em Annales de l'Institut Henri Poincar\'{e} Probabilit\'{e}s et
  Statistiques}, 46(1):45--58, 2010.

\bibitem[NPS13]{nourdin_multi-dimensional_2013}
Ivan Nourdin, Giovanni Peccati, and Roland Speicher.
\newblock Multi-dimensional {Semicircular} {Limits} on the {Free} {Wigner}
  {Chaos}.
\newblock In Robert~C. Dalang, Marco Dozzi, and Francesco Russo, editors, {\em
  Seminar on {Stochastic} {Analysis}, {Random} {Fields} and {Applications}
  {VII}}, number~67 in Progress in {Probability}, pages 211--221. Springer
  Basel, January 2013.

\bibitem[NS06]{nica_lectures_2006}
Alexandru Nica and Roland Speicher.
\newblock {\em Lectures on the combinatorics of free probability}, volume 335
  of {\em London {Mathematical} {Society} {Lecture} {Note} {Series}}.
\newblock Cambridge University Press, Cambridge, 2006.

\bibitem[NT14]{nourdin_central_2014}
Ivan Nourdin and Murad~S. Taqqu.
\newblock Central and non-central limit theorems in a free probability setting.
\newblock {\em Journal of Theoretical Probability}, 27(1):220--248, 2014.

\bibitem[Pec11]{peccati_chen-stein_2011}
Giovanni Peccati.
\newblock The {Chen}-{Stein} method for {Poisson} functionals.
\newblock {\em arXiv:1112.5051}, December 2011.

\bibitem[PSTU10]{peccati_steins_2010}
G.~Peccati, J.~L. Sol\'{e}, M.~S. Taqqu, and F.~Utzet.
\newblock Stein's method and normal approximation of {Poisson} functionals.
\newblock {\em The Annals of Probability}, 38(2):443--478, 2010.

\bibitem[PT05]{peccati_gaussian_2005}
Giovanni Peccati and Ciprian~A. Tudor.
\newblock Gaussian limits for vector-valued multiple stochastic integrals.
\newblock In {\em S\'{e}minaire de {Probabilit\'{e}s} {XXXVIII}}, volume 1857
  of {\em Lecture {Notes} in {Math}.}, pages 247--262. Springer, Berlin, 2005.

\bibitem[PT13]{peccati_gamma_2013}
Giovanni Peccati and Christoph Th\"{a}le.
\newblock Gamma limits and $u$-statistics on the {Poisson} space.
\newblock {\em ALEA. Latin American Journal of Probability and Mathematical
  Statistics}, 10(1):525--560, 2013.

\bibitem[PZ10]{peccati_multi-dimensional_2010}
Giovanni Peccati and Cengbo Zheng.
\newblock Multi-dimensional {Gaussian} fluctuations on the {Poisson} space.
\newblock {\em Electronic Journal of Probability}, 15:no. 48, 1487--1527, 2010.

\bibitem[Tao12]{tao_topics_2012}
Terence Tao.
\newblock {\em Topics in random matrix theory}, volume 132 of {\em Graduate
  {Studies} in {Mathematics}}.
\newblock American Mathematical Society, Providence, RI, 2012.

\bibitem[VDN92]{voiculescu_free_1992}
D.~V. Voiculescu, K.~J. Dykema, and A.~Nica.
\newblock {\em Free random variables}, volume~1 of {\em {CRM} {Monograph}
  {Series}}.
\newblock American Mathematical Society, Providence, RI, 1992.
\newblock A noncommutative probability approach to free products with
  applications to random matrices, operator algebras and harmonic analysis on
  free groups.

\end{thebibliography}

\end{document}